\documentclass[12pt]{amsart}
\usepackage{fullpage,url,amssymb,enumerate,colonequals}

\usepackage{mathrsfs} % for \mathscr (script letters)
\usepackage{MnSymbol}
\usepackage{extarrows}
\usepackage{lscape}
\usepackage[OT2,T1]{fontenc}
\usepackage{color}

\usepackage[
        colorlinks, citecolor=darkgreen,
        backref,
        pdfauthor={Alain Kraus}, % add other authors
]{hyperref}

\usepackage[french]{babel}

\newcommand{\A}{\mathfrak{A}}
\newcommand{\F}{\mathbb{F}}

\newcommand{\Q}{\mathbb{Q}}

\newcommand{\Z}{\mathbb{Z}}
\newcommand{\N}{\mathbb{N}}

\newcommand{\sym}{\mathfrak{S}}

\newcommand{\fq}{\mathfrak{q}}

\newcommand{\calC}{\mathcal{C}}

\newcommand{\calS}{\mathcal{S}}

\newcommand{\fp}{\mathfrak{p}}
\newcommand{\fP}{\mathfrak{P}}

\DeclareMathOperator{\Gal}{Gal}

\DeclareMathOperator{\Norm}{N}

\DeclareMathOperator{\Tr}{Tr}

\numberwithin{equation}{section}
\newtheorem{theorem}{Th\'eor\`eme}
\newtheorem{lemme}{Lemme}
\newtheorem{corollary}{Corollaire}
\newtheorem{proposition}{Proposition}

\theoremstyle{definition}

\newtheorem{question}[equation]{Question}

\theoremstyle{remark}
\newtheorem{remark}[equation]{Remarque}

\definecolor{darkgreen}{rgb}{0,0.5,0}

\newcommand \DB[2]{{#2}}
\newcommand{\Id}{\mathrm Id}

\begin{document}

\title{Id\'eaux premiers totalement d\'ecompos\'es   et sommes de Newton }
%\subjclass[2010]{Primary 11G30; Secondary 14G25, 14G40, 14K15, 14K20}

\author{Dominique Bernardi et Alain Kraus}

\begin{abstract}  Soient $K$ un corps de nombres et $f\in K[X]$ un polyn\^ome irr\'eductible unitaire   \`a  coef\-ficients 
dans l'anneau d'entiers $O_K$ de $K$. On se propose d'expliciter un crit\`ere effectif, en termes  du groupe de Galois de $f$ sur $K$ et d'une suite r\'ecurrente lin\'eaire associ\'ee \`a $f$,   
permettant parfois  de caract\'eriser les id\'eaux premiers  de $O_K$ modulo lesquels  $f$ est  totalement d\'ecompos\'e. 
Si $\alpha$ est une racine de $f$, ce  crit\`ere fournit donc une caract\'erisation   des id\'eaux premiers de $O_K$ qui sont totalement d\'ecompos\'es dans $K(\alpha)$.  Il s'applique en particulier si le  degr\'e de $f$ est au moins $4$ et le groupe de Galois de $f$ est le groupe sym\'etrique  ou le groupe altern\'e. 
\end{abstract}
\bigskip

\date{\today}

\address{Sorbonne Universit\'e,
Institut de Math\'ematiques de Jussieu - Paris Rive Gauche,
UMR 7586 CNRS - Paris Diderot,
4 Place Jussieu, 75005 Paris, 
France}
\email{alain.kraus@imj-prg.fr}
\email{dominique.bernardi@imj-prg.fr}

\renewcommand{\keywordsname}{Mots-cl\'es}

\makeatletter
\@namedef{subjclassname@1991}{MSC 2010}
\makeatother

\maketitle

\section{Introduction}   
Soit $K$ un corps de nombres\DB{ contenu dans une cl\^oture alg\'ebrique $\overline \Q$ de $\Q$}{}. Notons $O_K$ son anneau d'entiers.  Soit $f\in O_K[X]$ un polyn\^ome unitaire irr\'eductible sur $K$\DB{}{, dont le degr\'e sera not\'e $k$}. On se pr\' eoccupe dans cet article de la question suivante (cf. \cite{Weinstein} Question  A et  \cite{Dalawat} p. 27, dans le cas o\`u $K=\Q$) :

\begin{question}  \label{Q:question1} Comment caract\'eriser les id\'eaux premiers $\fp$ de $O_K$  tels que le polyn\^ome $f$ soit totalement d\'ecompos\'e  modulo $\fp$ i.e. que $f$ modulo $\fp$ soit s\'eparable et ait toutes ses racines dans $O_K/\fp$ ?
\end{question}

Une r\'eponse attendue consiste \`a expliciter  une caract\'erisation, ind\'ependante de tout  algorithme de factorisation de $f$ modulo $\fp$, permettant d'identifier ces id\'eaux premiers. 
On obtient alors, ce que l'on appelle parfois,  une loi de r\'eciprocit\'e pour $f$ (cf. \cite{Weinstein} p. 1, \cite{Dalawat} p. 32 et \cite{Wyman} p. 585).

Par exemple,  si  $K=\Q$,  une caract\'erisation   est la suivante.  Pour tout nombre premier $p$,  notons $N_p(f)$ le nombre de racines de $f$ modulo $p$. Soit $x$ la classe de $X$ dans la $\Z$-alg\`ebre $\Z[X]/(p,f)$.  Comme cons\'equence du th\'eor\`eme chinois, si $p$ ne divise pas le discriminant de $f$, on a l'\'equivalence
\begin{equation}
\label{(1.2)}
N_p(f)=k\  \   \Leftrightarrow \  x^p=x.
\end{equation}

La question~\ref{Q:question1} peut se reformuler comme suit (\cite{Weinstein}, Question B, p. 6) : 
soit $L/K$ une extension finie.  Comment  caract\'eriser les id\'eaux premiers de $O_K$ qui sont totalement d\'ecompos\'es dans $L$ ?
Si l'extension $L/K$ est ab\'elienne,  la th\'eorie du corps de classes fournit une r\'eponse  compl\`ete \`a cette question (voir par exemple \cite{Neu}).  Dans le cas g\'en\'eral, une approche  de ces questions  repose sur l'\'etude des repr\'esentations galoisiennes de degr\'e au moins $2$ du groupe de Galois $\Gal(\overline \Q/K)$.  
On pourra trouver dans  \cite{Weinstein}  l'\'etat des  connaissances sur ces questions.

Notre approche ici est beaucoup plus basique. Soit  $\Gal(f)$ le groupe de Galois de $f$ sur $K$. Pour tout id\'eal premier non nul $\fp$ de $O_K$  notons  $\Norm(\fp)$ le cardinal de $O_K/\fp$.
On    associe \`a $f$ la  suite r\'ecurrente lin\'eaire $(T_n)_{n\in \N}$ d'\'el\'ements de $O_K$,  dont le terme  $T_n$ est  la $n$-i\`eme somme de Newton des racines de $f$ dans $\overline {\DB\Q K}$.  On obtient un crit\`ere portant sur  $\Gal(f)$, impliquant le fait que 
pour tout id\'eal premier $\fp$ de $O_K$, sauf peut-\^etre un nombre fini,   $f$ est  totalement d\'ecompos\'e  modulo $\fp$ si et seulement si   $T_{\Norm(\fp)+1}-T_2$ appartient \`a $\fp$  (th\'eor\`eme~\ref{T:thm1}). 

\DB{Notons $k$}{Rappelons que $k$ est} le degr\'e de $f$. Si on a $k\geq 4$,   ce crit\`ere s'applique notamment si $\Gal(f)$ 
est isomorphe au groupe sym\'etrique $\sym_k$ ou \`a son sous-groupe altern\'e $\A_k$ (th\'eor\`eme~\ref{T:thm2}). Dans le cas o\`u   $\Gal(f)$  n'est pas isomorphe \`a $\sym_k$ ni $\A_k$, il importe dans l'utilisation du th\'eor\`eme~\ref{T:thm1} de d\'ecrire l'action de $\Gal(f)$ sur les racines de $f$. 
Le th\'eor\`eme~\ref{T:thm3} fournit un crit\`ere  qui permet  parfois de se dispenser de cette description. 

\' A titre indicatif, si $K=\Q$ et $f=X^5-X-1$, dont le groupe de Galois est $\sym_5$,  on obtient l'\'enonc\'e suivant (proposition~\ref{P:prop4}, assertion 1). Soit $(T_n)_{n\in \N}$ la suite d\'efinie par les \'egalit\'es
$$T_0=5,\quad T_1=T_2=T_3=0,\quad T_4=4  \quad \text{et} \quad T_{n+5}=T_{n+1}+T_n \quad \text{pour tout}\ n\in \N.$$
Alors, si $p\not\in \lbrace 2,5\rbrace$, on a $N_p(f)=5$ si et seulement si $T_{p+1}\equiv 0\pmod p$.
 
Par ailleurs, soit $H_K$ le corps de classes de Hilbert de $K$, qui est l'extension ab\'elienne  non ramifi\'ee maximale de $K$. D'apr\`es la  th\'eorie du corps de classes, pour qu'un id\'eal premier  de $O_K$ soit  principal il faut et il suffit qu'il soit  totalement d\'ecompos\'e dans $H_K$ (\cite{Neu}, (8.5) Corollary, p. 107). Supposons connu    le polyn\^ome minimal $f$ d'un \'el\'ement primitif  entier  de l'extension $H_K/K$. Si $(T_n)_{n\in \N}$ est la suite r\'ecurrente lin\'eaire associ\'ee \`a $f$ comme pr\'ec\'edemment, 
les r\'esultats obtenus  permettent alors  de caract\'eriser les id\'eaux premiers $\fp$ de $O_K$ qui sont principaux, en terme de la diff\'erence $T_{\Norm(\fp)+1}-T_2$. 
On pourra trouver deux illustrations    num\'eriques \`a ce sujet dans le paragraphe~\ref{S:par9}.

Tous les calculs num\'eriques que ce travail a n\'ecessit\'es ont \'et\'e effectu\'es \`a l'aide des logiciels de calculs {\tt Pari-gp} (\cite{Pari}) et {\tt Magma} (\cite{MAGMA}).

\section{\'Enonc\' es des r\'esultats}   
Soit $K$ un corps de nombres d'anneau d'entiers $O_K$. Soient  $k\geq 2$ un entier  et 
$$f=X^k-a_{k-1}X^{k-1}-a_{k-2}X^{k-2}-\cdots -a_1X-a_0\in O_K[X]$$
un polyn\^ome irr\'eductible sur $K$\DB{, de degr\'e $k$  \`a coefficients dans $O_K$}{}. 
Choisissons une num\'erotation  $\alpha_1,\cdots,\alpha_k$  des racines de $f$ dans $\overline \Q$. Posons  $G=\Gal(f)$ le groupe de Galois de $f$ sur $K$ i.e. 
le groupe de Galois de l'extension $K(\alpha_1,\cdots,\alpha_k)/K$.
On identifiera dans toute la suite $G$ avec le sous-groupe du groupe sym\'etrique $\sym_k$ associ\'e \`a l'action de $G$ sur
les racines de $f$. Avec cette identification, on a donc pour tout $\sigma\in G$ et tout $i\in \lbrace1,\cdots,k\rbrace$  l'\'egalit\'e
$$\sigma(\alpha_i)=\alpha_{\sigma(i)}.$$

\subsection{La suite \texorpdfstring{$(T_n)_{n\in \N}$}{Tn}} 
%On note  $(T_n)_{n\in \N}$   la suite d'\'el\'ements $O_K$ d\'efinie par les \'egalit\'es
%\begin{equation}
%\label{(2.2)}
%T_j=\sum_{i=1}^k \alpha_i^j \quad  \text{pour} \quad 0\leq j\leq k-1,
% \end{equation}
% \begin{equation}
% \label{(2.3)}
%T_{n+k}=a_{k-1}T_{n+k-1}+a_{k-2}T_{n+k-2}+\cdots +a_1T_{n+1}+a_0T_n\quad  \text{pour tout}\ n\in \N .
% \end{equation}
%Pour tout $n\in \N$ on a donc (cf. \cite{Ros},  \'egalit\'e (1) p. 304), 
%\begin{equation}
%\label{(2.4)}
%T_n=\sum_{i=1}^k \alpha_i^n.
% \end{equation}
%Par d\'efinition, $T_n$ est  la $n$-i\`eme somme de Newton des racines de  $f$.

%DB
Pour tout entier $n\in\N$, on note
\begin{equation}
 \label{(2.1)}
T_n=\sum_{i=1}^k\alpha_i^n
\end{equation}
la $n$-i\`eme somme de Newton des racines de $f$. La suite $T=(T_n)_{n\in\N}$ est d\'etermin\'ee par ses $k$ premiers termes et v\'erifie la r\'ecurrence lin\'eaire
$$T_{n+k}=a_{k-1}T_{n+k-1}+a_{k-2}T_{n+k-2}+\cdots +a_1T_{n+1}+a_0T_n$$
pour $n\in\N$. Tous ses termes appartiennent \`a $O_K$.

\subsection{L'entier  \texorpdfstring{$B_f$\DB{de $O_K$}{}}{Bf}} Pour tout $\sigma\in G$, \DB{distinct de l'identit\'e, posons}{notons}
\begin{equation}
 \label{(2.2)}
c(\sigma)=\sum_{i=1}^k \alpha_i\alpha_{\sigma(i)}-T_2.
 \end{equation}
Posons \DB{par ailleurs}{alors}
\begin{equation}
 \label{(2.3)}
B_f= \prod_{\sigma\in G, \sigma\neq \Id} c(\sigma).
\end{equation}

\begin{lemme} \label{L:lemme1}
L'\'el\'ement $B_f$ appartient \`a $O_K$. 
\end{lemme}

  \begin{proof} Soit $\tau$ un \'el\'ement de $G$. Pour tout $\sigma\in G$, $\sigma\neq\Id$, on a  
  $\tau \left(c(\sigma)\right)=c(\tau \sigma \tau^{-1})$, d'o\`u l'\'egalit\'e
 $$\tau(B_f)=\prod_{\sigma\in G, \sigma\neq\Id} c(\tau \sigma \tau^{-1}).$$
 L'application qui \`a $\sigma$ associe $\tau\sigma\tau^{-1}$ est une \DB{bijection de $G\backslash \lbrace\Id\rbrace$}{permutation de $G\setminus\{\Id\}$}. Il en r\'esulte  que   $B_f$ est fix\'e par $\tau$, d'o\`u  le r\'esultat.
 \end{proof}

\subsection{Le crit\`ere de d\'ecomposition} Soit $\Delta_f$ le discriminant de $f$. Pour tout id\'eal premier non nul $\fp$ de $O_K$ d\'esignons par $\Norm(\fp)$ le cardinal de $O_K/\fp$ i.e. la norme de $K$ sur $\Q$ de $\fp$.  Notons $N_{\fp}(f)$ le nombre de racines de $f$ modulo $\fp$.

\begin{theorem} \label{T:thm1}   Soit $\fp$ un id\'eal premier de $O_K$. Supposons que l'on ait
\begin{equation}
   \label{(2.4)}
\Delta_f B_f\not\equiv 0 \pmod {\fp}.
\end{equation}
 Alors, on a l'\'equivalence
 \begin{equation}
\label{(2.5)}
N_{\fp}(f)=k \   \Leftrightarrow \  T_{\Norm(\fp)+1}\equiv T_2\pmod {\fp}.
\end{equation}
  \end{theorem}
 \smallskip
Si  $B_f$ n'est pas nul, on obtient ainsi une loi de r\'eciprocit\'e pour $f$. Tel est le cas dans la situation  suivante.
\smallskip
\begin{theorem} \label{T:thm2} Supposons $k\geq 4$ et $G=\sym_k$ ou $G=\A_k$. Alors, on a $B_f \neq 0$.
 \end{theorem}
 
 \begin{remark} L'\'enonc\'e pr\'ec\'edent n'est pas valable pour  $k=3$. Dans ce cas,  on a $B_f\neq 0$ si et seulement si  $a_2^2+3a_1\neq 0$  (voir la d\'emonstration de la proposition~\ref{P:prop2}).   
De m\^eme, si  on a $k\geq 4$ et si $G$ n'est pas le groupe  $\sym_k$ ni le groupe $\A_k$, $B_f$ peut  \^etre nul. Par exemple, pour 
$f=X^4+X^3+X^2+X+1\in \Q[X],$
on a $G=\Z/4\Z\subseteq \sym_4$ et  $B_f=0$.  Cela \'etant,  pour ce polyn\^ome, il est bien connu que l'on a $N_p(f)=4$ si et  seulement si $5$ divise $p-1$. 
\end{remark} 

\subsection{Sur la factorisation de \texorpdfstring{$B_f$}{Bf}} La norme de $K$ sur $\Q$ de $B_f$ est un entier rationnel dont le nombre de chiffres d\'ecimaux est souvent tr\`es grand, notamment si $k\geq 6$. 
Il est alors tr\`es difficile et g\'en\'eralement prohibitif  de d\'eterminer ses diviseurs premiers. Dans certains cas, on peut n\'eanmoins y parvenir en remarquant que $B_f$ admet une factorisation qui est reli\'ee aux classes de conjugaison \DB{non triviales de $G$ i.e. distinctes de la classe de l'identit\'e}{de $G$}. Plus pr\'ecis\'ement, 
pour toute classe de conjugaison non triviale $\calC$  de $G$, posons 
\begin{equation}
 \label{(2.7)}
B_{\calC}= \prod_{\sigma\in \calC} c(\sigma).
\end{equation}
Pour tous $\tau\in G$ et $\sigma\in \calC$, l'\'egalit\'e $\tau \left(c(\sigma)\right)=c(\tau \sigma \tau^{-1})$, entra\^ine que $B_{\calC}$ est fix\'e par $\tau$, d'o\`u 
\begin{equation}
 \label{(2.8)}
B_{\calC}\in O_K.
\end{equation}
L'\'egalit\'e
\begin{equation}
\label{(2.9)}
B_f=\prod_{\calC} B_{\calC},
\end{equation}
o\`u $\calC$ parcourt l'ensemble des classe de conjugaison non triviales  de $G$, fournit alors une factorisation de $B_f$. Par ailleurs,  l'\'enonc\'e suivant facilite parfois 
cette factorisation.

\begin{proposition} \label{P:prop1}  Soient $\calC$ une classe de conjugaison non triviale  de $G$ qui n'est pas form\'ee d'\'el\'ements d'ordre $2$  et $\calC'$ la classe de conjugaison form\'ee des inverses des \'el\'ements de $\calC$.
\begin{itemize}
\item[1)]  Si \DB{pour tout $\sigma\in \calC$, $\sigma^{-1}$ est  dans $\calC$}{$\calC=\calC'$}, alors $B_{\calC}$ est un carr\'e dans $O_K$. 
\item[2)] \DB{Supposons qu'il existe $\sigma\in \calC$ tel que $\sigma^{-1}$ ne soit pas dans $\calC$. Soit $\calC'$ la classe de conjugaison de $\sigma^{-1}$
. Alors,}{Si $\calC\neq\calC'$, alors} on a $B_{\calC}=B_{\calC'}$. 
\end{itemize}
\end{proposition}

\begin{corollary} Le produit  des $B_{\calC}$, o\`u $\calC$ parcourt l'ensemble des classes de conjugaison non triviales de $G$ qui ne sont pas form\'ees d'\'el\'ements d'ordre 2, est un carr\'e dans $O_K$. 
\end{corollary}

\begin{proof} C'est une cons\'equence directe de la proposition~\ref{P:prop1}.
\end{proof}

\subsection{Une propri\'et\'e de divisibilit\'e  de \texorpdfstring{$B_f$}{Bf}} 

Si $G$ n'est pas $\sym_k$ ni $\A_k$, le calcul de $B_f$ n\'ecessite a priori la description explicite de l'action de $G$ sur les racines de $f$.
On peut parfois s'en affranchir en proc\'edant comme suit. 
 
Pour toute classe de conjugaison $\calS$ non triviale de $\sym_k$, consid\'erons le polyn\^ome $F_{\calS}\in \Z[X_1,\cdots,X_k]$ d\'efini par 
\begin{equation} 
\label{(2.10)}
F_{\calS}=\prod_{t\in \calS} \left(\sum_{i=1}^k X_iX_{t(i)}-N_2\right) \quad  \text{o\`u}\quad  N_2=\sum_{i=1}^k X_i^2.
\end{equation}
 On v\'erifie que $F_{\calS}$ est invariant sous l'action naturelle de $\sym_k$ sur $\Z[X_1,\cdots,X_k]$ (lemme~\ref{L:lemme6}). Par suite, $F_S$ est un polyn\^ome \`a coefficients dans $\Z$ en les  polyn\^omes  sym\'etriques \'el\'ementaires de $\Z[X_1,\cdots,X_k]$.  Il en r\'esulte que 
\begin{equation}
\label{(2.11)}
F_{\calS}(\alpha_1,\cdots, \alpha_k)\in O_K.
\end{equation}

\begin{theorem} \label{T:thm3}    On a dans $O_K$ la congruence
\begin{equation}
\label{(2.12)}
\prod_{\calS} F_{\calS}(\alpha_1,\cdots, \alpha_k)\equiv 0 \pmod{B_f},
\end{equation}
o\`u $\calS$ parcourt l'ensemble des classes de conjugaison non triviales de 
$\sym_k$. En particulier,
on a l'implication 
\begin{equation}
\label{(2.13)}
\prod_{\calS} F_{\calS}(\alpha_1,\cdots, \alpha_k) \neq 0  \   \Rightarrow\  B_f\neq 0.
\end{equation}
Pour tout   id\'eal premier $\fp$ de $O_K$ tel que 
\begin{equation}
\label{(2.14)}
 \prod_{\calS} F_{\calS}(\alpha_1,\cdots,\alpha_k)\not\equiv 0 \pmod {\fp},
 \end{equation}
on a l'\'equivalence
\begin{equation}
\label{(2.15)}
N_{\fp}(f)=k \   \Leftrightarrow \  T_{\Norm(\fp)+1}\equiv T_2\pmod {\fp}.
\end{equation}
\end{theorem}

Si l'on parvient \`a d\'eterminer les diviseurs premiers du produit des  $F_{\calS}(\alpha_1,\cdots, \alpha_k)$, on obtient alors une loi de r\'eciprocit\'e pour $f$
sans calculer $B_f$ explicitement.
 
 \subsection{Cons\'equences} Pr\'ecisons les \'enonc\'es pr\'ec\'edents dans quelques cas particuliers. 

\begin{proposition} \label{P:prop2} Soit $\fp$ un id\'eal premier de $O_K$. 

\begin{itemize}
\item[1)] Si $k=2$ et $\Delta_f\not\equiv 0 \pmod{\fp}$,
on a 
$$N_{\fp}(f)=2 \   \Leftrightarrow \  T_{\Norm(\fp)+1}\equiv a_1^2+2a_0\pmod {\fp}.$$ 
\item[2)] Si $k=3$ et 
$(a_2^2+3a_1)\Delta_f\not\equiv 0 \pmod{\fp}$, on a 
$$N_{\fp}(f)=3 \   \Leftrightarrow \  T_{\Norm(\fp)+1}\equiv a_2^2+2a_1\pmod {\fp}.$$ 
\end{itemize}
\end{proposition}
\smallskip

\begin{proposition} \label{P:prop3} Soit $\fp$ un id\'eal premier de $O_K$. Supposons que $f$ soit un  trin\^ome de la forme 
$$f=X^4-a_1X-a_0.$$
Si   $2a_1a_0\Delta_f\not\equiv 0 \pmod{\fp}$, on a 
$$N_{\fp}(f)=4 \   \Leftrightarrow \  T_{\Norm(\fp)+1}\equiv 0\pmod {\fp}.$$ 
\end{proposition}

En particulier, si $K=\Q$ et $f=X^4-X-1$,   pour tout nombre premier $p\neq 283$, on a 
$$N_p(f)=4 \   \Leftrightarrow \  T_{p+1}\equiv 0\pmod p.$$

Selmer a d\'emontr\'e que pour tout $n\geq 2$ le groupe de Galois sur $\Q$ du polyn\^ome $X^n-X-1\in \Q[X]$  est le groupe sym\'etrique $\sym_n$ (\cite{Selmer}). Dans le cas o\`u 
 $n\in \lbrace5,6,7\rbrace$, on a 
l'\'enonc\'e suivant.
 
 \smallskip
\begin{proposition} \label{P:prop4}  Supposons $K=\Q$. Soit $p$ un nombre premier. 

\begin{itemize} 

\item[1)] Si $f=X^5-X-1$ et  $p\not\in \lbrace 2,5\rbrace$, on a 
$$N_{p}(f)=5 \   \Leftrightarrow \  T_{p+1}\equiv 0\pmod p.$$
\item[2)]  Si $f=X^6-X-1$ et  $p\not\in \lbrace 2,3,7,13,1663 \rbrace$, on a 
$$N_{p}(f)=6 \   \Leftrightarrow \  T_{p+1}\equiv 0\pmod p.$$
\item[3)]  Si $f=X^7-X-1$ et  $p\not\in \lbrace 2,3, 7, 15961129   \rbrace$, on a 
$$N_{p}(f)=7 \   \Leftrightarrow \  T_{p+1}\equiv 0\pmod p.$$

\end{itemize}
\end{proposition}

\begin{remark} En ce qui concerne  le polyn\^ome $f=X^8-X-1$, si $\calC$ est la classe de conjugaison des $8$-cycles dans $\sym_8$, la racine carr\'ee de   $B_{\calC}$   est un entier  poss\'edant 830 chiffres d\'ecimaux. Il semble 
difficile d'obtenir   sa factorisation  compl\`ete en produit de nombres premiers. 
\end{remark} 
\smallskip

\begin{proposition} \label{P:prop5} Supposons $K=\Q$ et $f=X^5-2X^4+2X^3-X^2+1$. Alors, le groupe de Galois  de $f$ est   di\'edral   d'ordre $10$.
Pour tout nombre premier $p\not\in \lbrace 3,47\rbrace$, on a  
$$N_{p}(f)=5 \   \Leftrightarrow \  T_{p+1}\equiv 0\pmod p.$$
\end{proposition}

Terminons ce paragraphe avec l'exemple suivant. 

\begin{proposition} \label{P:prop6}  Supposons $K=\Q(\alpha)$ avec $\alpha^3-\alpha-1=0$. Posons 
$$u=-\alpha^2+\alpha+2 \quad \text{et}\quad f=X^5+u^3X+u\in  K[X].$$
Alors, $f$ est irr\'eductible sur $K$ et $\Gal(f)=\sym_5$. Soient $\fp_5$ et $\fp_{61}$ les id\'eaux premiers de $O_K$  au-dessus de $5$ et $61$  qui sont de degr\'e r\'esiduel $1$. Pour tout id\'eal premier $\fp$ de $O_K$,  distinct de $\fp_5$ et $\fp_{61}$,  on a 
\begin{equation}
\label{(2.17)}
N_{\fp}(f)=5 \   \Leftrightarrow \ T_{\Norm(\fp)+1}\equiv 0\pmod {\fp}.
\end{equation}
\end{proposition}

\section{D\'emonstration du th\'eor\`eme~\ref{T:thm1}}
Posons  $\widetilde{K}=K(\alpha_1,\cdots,\alpha_k)$ et notons $O_{\widetilde{K}}$ son anneau d'entiers. Rappelons que $G= \Gal(\widetilde{K}/K)$.

Le lemme qui suit ne d\'epend pas de l'hypoth\`ese faite sur  $B_f$.

 \begin{lemme} \label{L:lemme2} Supposons $N_{\fp}(f)=k$. Alors, on a  $T_{\Norm(\fp)+1}\equiv T_2\pmod {\fp}.$
 \end{lemme}
 
% \begin{proof}     Parce que $\fp$ ne divise pas $\Delta_f$ et que $N_{\fp}(f)=k$,  $\fp$ est totalement d\'ecompos\'e dans  $\widetilde{K}$ (\cite{Cohen}, p.  79 et Prop. 2.3.9). Soit $\fP$ un id\'eal premier de $O_{\widetilde{K}}$ au-dessus de $\fp$. 
%La substitution de Frobenius $\Frob_{\fP/\fp}\in G$ est donc  l'identit\'e. Autrement dit, pour tout $i=1,\cdots,k$, on a 
% $$\alpha_i=\Frob_{\fP/\fp}(\alpha_i)\equiv \alpha_i^{\Norm \fp}\pmod{\fP}.$$
% D'apr\`es l'\'egalit\'e \eqref{(2.4)}, on a 
% $$T_{\Norm(\fp)+1}=\alpha_1^{\Norm(\fp)+1}+\cdots+\alpha_k^{\Norm(\fp)+1},$$
% d'o\`u 
% $$T_2= \alpha_1^2+\cdots +\alpha_k^2 \equiv T_{\Norm(\fp)+1} \pmod{\fP}.$$
% Cela implique l'assertion car $T_2$ et $T_{\Norm(\fp)+1}$ sont dans $O_K$.
% \end{proof} 

\begin{proof}
Par hypoth\`ese, il existe $k$ \'el\'ements $\beta_1\dots\beta_k$ de $O_K$ tels que $f\equiv\prod_{i=1}^k(X-\beta_i)\pmod\fp$. Si $(V_n)_{n\in\N}$ est la suite des sommes de Newton des $\beta_i$, on en d\'eduit que pour tout $n\in\N$ on a $T_n\equiv V_n\pmod\fp$. En cons\'equence, on a
\[T_{\Norm(\fp)+1}\equiv V_{\Norm(\fp)+1}\equiv V_2\equiv T_2\pmod\fp,\]
la congruence centrale venant de ce que l'on a $x^{\Norm(\fp)+1}\equiv x^2\pmod\fp$ pour tout $x\in O_K$.
\end{proof}

Inversement, supposons   $T_{\Norm(\fp)+1}\equiv T_2 \pmod{\fp}$ et montrons que l'on a $N_{\fp}(f)=k$. 
 
L'id\'eal   $\fp$ est non ramifi\'e dans $\widetilde{K}$  car il  ne divise pas $\Delta_f$. Soient $\fP$ un id\'eal premier de $O_{\widetilde{K}}$ au-dessus de $\fp$ et   $\sigma\in G$ la substitution de Frobenius en $\fP$.  Pour tout $i\in\{1,\cdots,k\}$, on a 
$$\sigma(\alpha_i)=\alpha_{\sigma(i)}\quad \text{et}\quad \alpha_i^{\Norm(\fp)}\equiv \alpha_{\sigma(i)} \pmod{\fP}.$$
On a donc dans  $O_{\widetilde{K}}$ la congruence (\'egalit\'e \eqref{(2.1)})
 $$T_{\Norm(\fp)+1}\equiv    \sum_{i=1}^k\alpha_i \alpha_{\sigma(i)} \pmod {\fP}.$$
Il en r\'esulte que l'on a 
$$\sum_{i=1}^k\alpha_i \alpha_{\sigma(i)}\equiv T_2 \pmod {\fP}.$$
V\'erifions que $\sigma$ est distinct de l'identit\'e. Supposons le contraire. Dans ce cas, on  d\'eduit    des \'egalit\'es  ~\eqref{(2.2)} et   \eqref{(2.3)} que   $\fP$ divise $B_f$. 
Parce que $B_f$  est   dans $O_K$ (lemme~\ref{L:lemme1}), on a donc 
 $$B_f\equiv 0 \pmod {\fp},$$
 ce qui contredit la condition~\eqref{(2.4)}, et prouve  notre assertion. 
Par suite,  $\fp$ est 
 totalement d\'ecompos\'e dans $\widetilde{K}$ ou ce qui revient au m\^eme dans $K$, d'o\`u  $N_{\fp}(f)=k$ (\cite{Cohen},  Prop. 2.3.9). 
 
 Cela  \'etablit l'\'equivalence~\eqref{(2.5)}  et  le th\'eor\`eme~\ref{T:thm1}.
 
\section{D\'emonstration du th\'eor\`eme~\ref{T:thm2}}
Consid\'erons d\'esormais un \'el\'ement  $\sigma\in G$, distinct de l'identit\'e. On va d\'emontrer que l'on a
\begin{equation}
  \label{(4.1)}
\sum_{i=1}^k \alpha_i\alpha_{\sigma(i)}\neq T_2,
\end{equation}
ce qui prouvera le r\'esultat. 
Pour cela,  on est amen\'e  \`a distinguer plusieurs cas en fonction de la d\'ecomposition de $\sigma$ en produit de cycles \`a supports disjoints. Dans chacun des cas, on va supposer que l'on a
\begin{equation}   \label{(4.2)} 
\sum_{i=1}^k \alpha_i\alpha_{\sigma(i)}=T_2
\end{equation}
  et aboutir \`a une contradiction.

 \begin{lemme}  \label{L:lemme3} Supposons   qu'il existe  une transposition dans la d\'ecomposition de $\sigma$ en produit de cycles \`a supports disjoints. Alors, la condition \eqref{(4.1)} est satisfaite.
 \end{lemme}

 \begin{proof}  
Soit $(\alpha_{i_1},\alpha_{i_2})$ une transposition intervenant dans la d\'ecomposition de $\sigma$ en produit de cycles \`a supports disjoints. 

Supposons  que l'on ait $\sigma=(\alpha_{i_1},\alpha_{i_2})$. On a les \'egalit\'es 
\begin{equation}
\label{(4.3)}
\sum_{i=1}^k \alpha_i\alpha_{\sigma(i)}=2\alpha_{i_1}\alpha_{i_2}+T_2-\alpha_{i_1}^2-\alpha_{i_2}^2=-(\alpha_{i_1}-\alpha_{i_2})^2+T_2.
\end{equation}
La condition \eqref{(4.2)} implique alors $\alpha_{i_1}=\alpha_{i_2}$ et une contradiction.

On a donc $\sigma\neq (\alpha_{i_1},\alpha_{i_2})$.  
Il  existe ainsi   dans la d\'ecomposition de $\sigma$ un cycle 
$$(\alpha_{i_3},\cdots,\alpha_{i_t})$$
de longueur au moins $2$, dont le support est disjoint de $\lbrace \alpha_{i_1},\alpha_{i_2}\rbrace$.
D'apr\`es la condition \eqref{(4.2)}, on a donc une \'egalit\'e de la forme
\begin{equation}
\label{(4.4)} 
2\alpha_{i_1}\alpha_{i_2}+(\alpha_{i_3}\alpha_{i_4}+\cdots+\alpha_{i_{t-1}}\alpha_{i_t}+\alpha_{i_t}\alpha_{i_3})+P(\cdots,\alpha_j,\cdots)=T_2,
\end{equation}
 o\`u $P$ est un polyn\^ome homog\`ene de degr\'e $2$ en les $\alpha_j$ avec $j\neq i_1,i_2,i_3,\cdots,i_t.$

Appliquons le $3$-cycle  $(i_1,i_2,i_3)\in \A_k\subseteq G$ aux deux membres de l'\'egalit\'e \eqref{(4.4)}. On a $t\geq 4$ et 
on obtient  ainsi
\begin{equation}
 \label{(4.5)}
2\alpha_{i_2}\alpha_{i_3}+(\alpha_{i_1}\alpha_{i_4}+\cdots+\alpha_{i_{t-1}}\alpha_{i_t}+\alpha_{i_t}\alpha_{i_1})+P(\cdots,\alpha_j,\cdots)=T_2.
\end{equation}
Par soustraction des \'egalit\'es  \eqref{(4.4)} et \eqref{(4.5)}, on a donc
$$2\alpha_{i_2}(\alpha_{i_1}-\alpha_{i_3})+\alpha_{i_4}(\alpha_{i_3}-\alpha_{i_1})+\alpha_{i_t}(\alpha_{i_3}-\alpha_{i_1})=0.$$
On a $\alpha_{i_1}\neq \alpha_{i_3}$, d'o\`u 
$$2\alpha_{i_2}=\alpha_{i_4}+\alpha_{i_t}.$$
En appliquant de nouveau le $3$-cycle  $(i_1,i_2,i_3)\in G$ aux deux membres de cette \'egalit\'e, on obtient  
$$2\alpha_{i_3}=\alpha_{i_4}+\alpha_{i_t},$$
d'o\`u   $\alpha_{i_2}=\alpha_{i_3}$, puis une contradiction
 et le r\'esultat.
 \end{proof}  

   \begin{lemme}  \label{L:lemme4} Soit $p\geq 4$ un entier. Supposons  qu'il existe   un $p$-cycle  dans la d\'ecomposition de $\sigma$ en produit de cycles \`a supports disjoints. 
Alors, la condition \eqref{(4.1)} est satisfaite.
  \end{lemme}

 \begin{proof}  Soit 
 $(\alpha_{i_1},\alpha_{i_2},\cdots,\alpha_{i_p})$ un $p$-cycle intervenant dans la d\'ecomposition de $\sigma$ en produit de cycles \`a supports disjoints. 
On a alors une expression de la forme
\begin{equation}
 \label{(4.6)}
\alpha_{i_1}\alpha_{i_2}+\alpha_{i_2}\alpha_{i_3}+\alpha_{i_3}\alpha_{i_4}+\cdots+\alpha_{i_p}\alpha_{i_1}+P(\cdots,\alpha_j,\cdots)=T_2,
\end{equation}
o\`u $P$ est un polyn\^ome homog\`ene de degr\'e $2$ en les $\alpha_j$ avec  $j\not\in \lbrace    i_1,i_2,\cdots,i_p\rbrace.$
En appliquant le $3$-cycle $(i_1,i_2,i_3)\in G$ aux deux membres de \eqref{(4.6)}, on obtient  
\begin{equation}
 \label{(4.7)}
\alpha_{i_2}\alpha_{i_3}+\alpha_{i_3}\alpha_{i_1}+\alpha_{i_1}\alpha_{i_4}+\cdots+\alpha_{i_p}\alpha_{i_2}+P(\cdots,\alpha_j,\cdots)=T_2.
 \end{equation}
   Par soustraction  de \eqref{(4.6)} et \eqref{(4.7)}, on a donc 
\begin{equation}
 \label{(4.8)}
\alpha_{i_1}(\alpha_{i_2}-\alpha_{i_3}+\alpha_{i_p}-\alpha_{i_4})+\alpha_{i_3}\alpha_{i_4}-\alpha_{i_p}\alpha_{i_2}=0. 
\end{equation}

1) Si $p=4$, on a alors
$$\alpha_{i_1}(\alpha_{i_2}-\alpha_{i_3})+\alpha_{i_4}(\alpha_{i_3}-\alpha_{i_2})=0,$$ 
d'o\`u  $\alpha_{i_2}=\alpha_{i_3}$ ou  $\alpha_{i_1}=\alpha_{i_4}$ et une contradiction. 

2) Si $p\geq 5$, en appliquant le $3$-cycle $(i_2,i_3,i_4)\in G$ \`a l'\'egalit\'e \eqref{(4.8)}, cela conduit \`a
\begin{equation}
 \label{(4.9)}
\alpha_{i_1}(\alpha_{i_3}-\alpha_{i_4}+\alpha_{i_p}-\alpha_{i_2})+\alpha_{i_4}\alpha_{i_2}-\alpha_{i_p}\alpha_{i_3}=0.
\end{equation}
Par soustraction de \eqref{(4.8)} et \eqref{(4.9)},  il vient
 $$\alpha_{i_1}(2\alpha_{i_2}-2\alpha_{i_3})+\alpha_{i_4}(\alpha_{i_3}-\alpha_{i_2})+\alpha_{i_p}(\alpha_{i_3}-\alpha_{i_2})=0.$$
On a $\alpha_{i_2}\neq \alpha_{i_3}$, d'o\`u 
$$2\alpha_{i_1}=\alpha_{i_4}+\alpha_{i_p}.$$
 En appliquant  le $3$-cycle $(i_1,i_2,i_3)\in  G$ \`a cette \'egalit\'e, on obtient
 $$2\alpha_{i_2}=\alpha_{i_4}+\alpha_{i_p},$$ 
puis $\alpha_{i_1}=\alpha_{i_2}$, d'o\`u une contradiction  et le r\'esultat.
  \end{proof}

Compte tenu des lemmes~\ref{L:lemme3} et \ref{L:lemme4}, il reste \`a \'etablir l'\'enonc\'e suivant.

   \begin{lemme}  \label{L:lemme5} Supposons   que $\sigma$ soit un produit de  $3$-cycles  \`a supports disjoints. Alors, la condition \eqref{(4.1)} est satisfaite.
  \end{lemme}

  \begin{proof}    
  1) Supposons que $\sigma$ soit un $3$-cycle. Posons
$\sigma=(\alpha_{i_1},\alpha_{i_2},\alpha_{i_3})$. Toutes les racines de $f$ autres  que $\alpha_{i_1},\alpha_{i_2},\alpha_{i_3}$ \'etant fix\'ees par $\sigma$, on a
\begin{equation}
 \label{(4.10)}
\alpha_{i_1}\alpha_{i_2}+\alpha_{i_2}\alpha_{i_3}+\alpha_{i_3}\alpha_{i_1}+\sum_{j\not\in \lbrace i_1,i_2,i_3\rbrace} \alpha_j^2=T_2.
\end{equation}
On a $k\geq 4$ et en  appliquant le cycle $(i_2,i_3,i_4)\in G$ aux deux membres de cette \'egalit\'e, on  obtient
 \begin{equation}
  \label{(4.11)}
\alpha_{i_1}\alpha_{i_3}+\alpha_{i_3}\alpha_{i_4}+\alpha_{i_4}\alpha_{i_1}+\alpha_{i_2}^2+\sum_{j\not\in \lbrace i_1,i_2,i_3,i_4\rbrace} \alpha_j^2=T_2.
\end{equation}
Par soustraction  de \eqref{(4.10)}  et  \eqref{(4.11)},  il vient
$$\alpha_{i_1}(\alpha_{i_2}-\alpha_{i_4})+\alpha_{i_3}(\alpha_{i_2}-\alpha_{i_4})+\alpha_{i_4}^2-\alpha_{i_2}^2=0.$$
On a $\alpha_{i_2}\neq\alpha_{i_4}$, d'o\`u 
$$\alpha_{i_1}+\alpha_{i_3}=\alpha_{i_2}+\alpha_{i_4}.$$
En appliquant le $3$-cycle $(i_1,i_2,i_3)$ aux deux membres de cette \'egalit\'e, on obtient
$$\alpha_{i_2}+\alpha_{i_1}=\alpha_{i_3}+\alpha_{i_4},$$
puis
$\alpha_{i_3}-\alpha_{i_2}=\alpha_{i_2}-\alpha_{i_3},$
d'o\`u  $\alpha_{i_2}=\alpha_{i_3}$ et une contradiction.

2) Supposons que $\sigma$ soit un  produit de au moins deux $3$-cycles \`a supports disjoints 
$$(\alpha_{i_1},\alpha_{i_2},\alpha_{i_3})\quad \text{et}\quad (\alpha_{i_4},\alpha_{i_5},\alpha_{i_6}).$$
  On a donc  une expression de la forme
  \begin{equation}
    \label{(4.12)}
  \alpha_{i_1}\alpha_{i_2}+\alpha_{i_2}\alpha_{i_3}+\alpha_{i_3}\alpha_{i_1}+\alpha_{i_4}\alpha_{i_5}+\alpha_{i_5}\alpha_{i_6}+\alpha_{i_6}\alpha_{i_4}+P(\cdots,\alpha_j,\cdots)=T_2,
  \end{equation}
 o\`u $P$ est un polyn\^ome homog\`ene de degr\'e $2$ en les $\alpha_j$ avec  $j\not\in \lbrace i_1,i_2,i_3,i_4,i_5,i_6\rbrace$. 
En appliquant le $3$-cycle $(i_2,i_3,i_4)\in G$ aux deux membres de \eqref{(4.12)}, on obtient   
  \begin{equation}
     \label{(4.13)} 
\alpha_{i_1}\alpha_{i_3}+\alpha_{i_3}\alpha_{i_4}+\alpha_{i_4}\alpha_{i_1}+\alpha_{i_2}\alpha_{i_5}+\alpha_{i_5}\alpha_{i_6}+\alpha_{i_6}\alpha_{i_2}+P(\cdots,\alpha_j,\cdots)=T_2.
  \end{equation}
   Par soustraction de  \eqref{(4.12)}  et  \eqref{(4.13)}, on a donc
 $$\alpha_{i_1}(\alpha_{i_2}-\alpha_{i_4})+\alpha_{i_3}(\alpha_{i_2}-\alpha_{i_4})+\alpha_{i_5}(\alpha_{i_4}-\alpha_{i_2})+\alpha_{i_6}(\alpha_{i_4}-\alpha_{i_2})=0.$$
On a $\alpha_{i_2}\neq \alpha_{i_4}$, d'o\`u 
 $$\alpha_{i_1}+\alpha_{i_3}=\alpha_{i_5}+\alpha_{i_6}.$$
En appliquant le $3$-cycle $(i_1,i_2,i_3)$ aux deux membres, il vient 
$$\alpha_{i_2}+\alpha_{i_1}=\alpha_{i_5}+\alpha_{i_6},$$
d'o\`u $\alpha_{i_3}=\alpha_{i_2}$. On a ainsi une contradiction et le r\'esultat.
\end{proof}    

Cela termine la d\'emonstration du th\'eor\`eme~\ref{T:thm2}.

\section{D\'emonstration de la proposition~\ref{P:prop1}}
 
Pour tout $\sigma\in G$, $\sigma\neq 1$,  on a 
\begin{equation}
 \label{(5.1)} 
c(\sigma)=c\left(\sigma^{-1}\right).
\end{equation}

D'apr\`es l'hypoth\`ese faite, pour tout $\sigma\in \calC$ on a $\sigma\neq \sigma^{-1}$. 
 Soit $X$ un syst\`eme de repr\'esentants des  paires $\lbrace \sigma, \sigma^{-1}\rbrace$ lorsque 
 $\sigma$ parcourt $\calC$ i.e. $X$ est une partie de $\calC$ contenant exactement un \'el\'ement des paires 
 $\lbrace \sigma, \sigma^{-1}\rbrace$ pour $\sigma\in \calC$. 
Posons
 $$x=\prod_{s\in X} c(s).$$
 D'apr\`es l'\'egalit\'e~\eqref{(5.1)}, le produit des $c(s)$ pour $s\in X$ ne d\'epend pas du syst\`eme de repr\'esentants $X$ choisi et on~a 
 \begin{equation}
 \label{(5.2)}
B_{\calC}=x^2.
\end{equation} 
 Par ailleurs, soit $\tau$ un \'el\'ement de $G$. Pour tout $s\in X$, on a $\tau\left(c(s)\right)=c\left(\tau s \tau^{-1}\right)$.  De plus, quand $s$ parcourt $X$, $\tau s\tau^{-1}$ 
 parcourt un autre syst\`eme de repr\'esentants. On a donc $\tau(x)=x$, par suite  $x$ est dans $O_K$, d'o\`u la premi\`ere assertion. 
 
 En ce qui concerne la seconde assertion, en posant $\calC=\lbrace s_1,\cdots,s_r\rbrace$, on a $\calC'=\lbrace s_1^{-1},\cdots,s_r^{-1}\rbrace$. Les \'egalit\'es~ \eqref{(2.7)} et~\eqref{(5.1)} impliquent alors le r\'esultat.
 
  \section{D\'emonstration du th\'eor\`eme~\ref{T:thm3}}
 
  Soit $\calS$ une classe de conjugaison non triviale de $\sym_k$. 
 
 \begin{lemme} \label{L:lemme6} Le polyn\^ome $F_{\calS}$ est invariant sous l'action de $\sym_k$ 
 \end{lemme}
 
 \begin{proof} Elle est analogue \`a  celle du lemme~\ref{L:lemme1}. En effet, soit $\tau\in \sym_k$. D'apr\`es \eqref{(2.10)}, on a les \'egalit\'es 
$$\tau.F_{\calS}=\prod_{t\in \calS}\left( \sum_{i=1}^k X_{\tau(i)} X_{\tau t(i)}-N_2\right)= \prod_{t\in \calS} \left(\sum_{j=1}^k X_j X_{\tau t \tau^{-1}(j)}-N_2\right).$$
 L'application de $\calS$ dans $\calS$ qui \`a $t$ associe $\tau t \tau^{-1}$ \'etant une bijection, on a donc   $\tau.F_{\calS}=F_{\calS}$.
 \end{proof} 
 
Rappelons que l'on a 
\begin{equation}
\label{(6.1)}
F_{\calS}(\alpha_1,\cdots,\alpha_k)=\prod_{\sigma\in \calS}\left(\sum_{i=1}^k \alpha_i\alpha_{\sigma(i)} -T_2\right).
\end{equation}
 Comme cons\'equence directe du lemme~\ref{L:lemme6}, on en d\'eduit l'\'enonc\'e qui suit i.e. la condition \eqref{(2.11)}.

\begin{corollary} \label{C:cor2}  L'\'el\'ement $F_{\calS}(\alpha_1,\cdots,\alpha_k)$ est dans $O_K$. 
\end{corollary}
 
 \begin{lemme} \label{L:lemme7}  Soit $\calS_0$  la classe de conjugaison des transpositions de $\sym_k$.  On a
\begin{equation}
\label{(6.2)}
F_{\calS_0}(\alpha_1,\cdots,\alpha_k)=(-1)^{\frac{k(k-1)}{2}}\Delta_f.
\end{equation}
 \end{lemme}
 
  \begin{proof} Il  y a $k(k-1)/2$ transpositions dans $\sym_k$ et on a (cf.~\eqref{(4.3)})
$$F_{\calS_0}(\alpha_1,\cdots,\alpha_k)=(-1)^{\frac{k(k-1)}{2}} \prod_{i<j} (\alpha_i-\alpha_j)^2=(-1)^{\frac{k(k-1)}{2}} \Delta_f.$$
  \end{proof} 
Le th\'eor\`eme~\ref{T:thm3} se d\'eduit comme suit.  D'apr\`es l'\'egalit\'e~\eqref{(6.1)}, on a 
$$\prod_{\calS} F_{\calS}(\alpha_1,\cdots,\alpha_k)=\prod_{\sigma\in \sym_k, \sigma\neq 1} \left(\sum_{i=1}^k \alpha_i\alpha_{\sigma(i)} -T_2\right).$$
On d\'eduit alors de~\eqref{(2.7)}-\eqref{(2.9)}  et du corollaire~\ref{C:cor2} que $B_f$ divise dans $O_K$ le produit des $F_{\calS}(\alpha_1,\cdots,\alpha_k)$,  d'o\`u  les  conditions \eqref{(2.12)} et \eqref{(2.13)}. Le th\'eor\`eme~\ref{T:thm1} et le lemme~\ref{L:lemme7}  entra\^inent alors le r\'esultat.

\section{Les \'el\'ements  \texorpdfstring{$F_{\calS}(\alpha_1,\cdots,\alpha_k)$  pour $k\in \lbrace 3,4\rbrace$ \label{S:par7}}{Fs pour k=3,4}}
Pour toute classe de conjugaison non triviale $\calS$ de $\sym_k$, posons par commodit\'e
\begin{equation}
\label{(7.1)}
R_{\calS}=F_{\calS}(\alpha_1,\cdots,\alpha_k).
\end{equation}
Si $\calS_0$ est la classe de conjugaison des transpositions, le  lemme~\ref{L:lemme7} permet de d\'eterminer simplement $R_{\calS_0}$. 
Dans le cas o\`u $k\in \lbrace 3,4\rbrace$ et    $\calS\neq \calS_0$, on   exprime ci-dessous $R_{\calS}$ 
en  un  polyn\^ome 
des fonctions   sym\'etriques \'el\'ementaires $e_1,\cdots,e_k$ des  racines $\alpha_i$ de $f$,
$$e_1=\sum_{i=1}^k \alpha_i,\quad e_2=\sum_{i<j} \alpha_i\alpha_j,\  \cdots .$$

\subsection{Cas o\`u \texorpdfstring{$k=3$}{k=3}} 
Soit  $\calS_1$ la classe des $3$-cycles de $\sym_3$. On  a les \'egalit\'es 
$$\alpha_1\alpha_2+\alpha_2\alpha_3+\alpha_3\alpha_1-T_2=\alpha_1\alpha_3+\alpha_2\alpha_1+\alpha_3\alpha_2-T_2=-e_1^2 + 3e_2,$$
d'o\`u 
\begin{equation}
\label{(7.2)}
R_{\calS_1}=(e_1^2-3e_2)^2.
\end{equation}

 \subsection{Cas o\`u \texorpdfstring{$k=4$}{k=4}} 
On v\'erifie avec {\tt Magma}   les \'egalit\'es suivantes.

Soit   $\calS_1$ la classe des doubles transpositions de $\sym_4$. On a 
\begin{equation}
\label{(7.3)}
R_{\calS_1}=-e_1^6 + 8e_1^4e_2 - 4e_1^3e_3 - 20e_1^2e_2^2 + 24e_1^2e_4 + 8e_1e_2e_3 + 16e_2^3
    - 64e_2e_4 + 8e_3^2.
\end{equation}   
Soit   $\calS_2$ la classe des $3$-cycles. En posant 
$$g=3e_3e_1^5 - (e_2^2 +3e_4)e_1^4 - 19e_3e_2e_1^3 + (6e_2^3 + 16e_4e_2 + 8e_3^2)e_1^2$$$$ + (30e_3e_2^2 + 8e_3e_4)e_1 -9e_2^4 - 24e_4e_2^2 - 24e_3^2e_2 - 16e_4^2,$$
on a 
\begin{equation}
\label{(7.4)}
R_{\calS_2}=g^2.
\end{equation} 
Soit   $\calS_3$ la classe des $4$-cycles. On a 
\begin{equation}
\label{(7.5)}
R_{\calS_3}=(e_1^6 - 8e_1^4e_2 + e_1^3e_3 + 21e_1^2e_2^2 - 3e_1^2e_4 - 3e_1e_2e_3 -
        18e_2^3 + 8e_2e_4 + e_3^2)^2.
\end{equation}

 \section{D\'emonstrations des propositions~\ref{P:prop2},  \ref{P:prop3}, \ref{P:prop4}, \ref{P:prop5} et \ref{P:prop6} }
 Reprenons sans autre pr\'ecision les notations utilis\'ees dans le paragraphe~\ref{S:par7}.
 
 \subsection{La proposition 2}   1) Supposons $k=2$. On a 
 $e_1=a_1$ et $e_2=-a_0,$
d'o\`u 
 $$T_2=\alpha_1^2+\alpha_2^2=(\alpha_1+\alpha_2)^2-2\alpha_1\alpha_2=e_1^2-2e_2=a_1^2+2a_0.$$
 Par ailleurs, on a
 $$B_f=2\alpha_1\alpha_2-T_2=-(\alpha_1-\alpha_2)^2=-\Delta_f,$$
 d'o\`u l'assertion (th\'eor\`eme~\ref{T:thm1}). 
 
2) Supposons $k=3$. On a $e_1=a_2$, $e_2=-a_1$ et $e_3=a_0$.   On a donc
$$T_2=a_2^2+2a_1.$$
D'apr\`es  le lemme~\ref{L:lemme7} et l'\'egalit\'e~\eqref{(7.2)}, on obtient
$$R_{\calS_0}R_{\calS_1}=-\Delta_f(a_2^2+3a_1)^2.$$
Le  th\'eor\`eme~\ref{T:thm3}  entra\^ine alors l'assertion. 
 
  \subsection{La proposition 3} On a 
 $$e_1=e_2=0 \quad e_3=a_1,\quad e_4=-a_0\quad \text{et}\quad T_2=0.$$
 D'apr\`es  le lemme~\ref{L:lemme7} et les \'egalit\'es~\eqref{(7.3)}-\eqref{(7.5)}, on a  
$$R_{\calS_0}R_{\calS_1}R_{\calS_2}R_{\calS_3} =2^{11}\Delta_fe_3^6e_4^4,$$
d'o\`u le r\'esultat (th\'eor\`eme~\ref{T:thm3}).

 \subsection{La proposition 4}  Posons $A=\Z[X]/(p,f)$. C'est un $\F_p$-espace vectoriel de dimension $k$. Notons $\Tr_{A/\F_p}$ la forme trace. On utilisera \`a de nombreuses reprises l'\'equivalence~\eqref{(1.2)} ainsi que   l'\'egalit\'e
 \begin{equation}
\label{(8.1)}
 \Tr_{A/\F_p}\left(x^{p+1}\right)=T_{p+1}+p\Z. 
 \end{equation}
 Rappelons que pour  chaque  exemple  consid\'er\'e   dans l'\'enonc\'e de cette proposition, on a $G=\sym_k$. Par ailleurs, on a 
 $T_2=0$. 
 
 1) On a $\Delta_f=19.151$.
 En d\'eterminant  une approximation  num\'erique des racines de $f$, on v\'erifie que l'on a
{\small{$$B_f=-2^4.5^{19}.7^2.19.151.467^2.761^2.2477.$$}}
Si $p\in \lbrace 19,151\rbrace$, $f$ n'est pas s\'eparable modulo $p$, donc $N_p(f)\neq 5$. On a $T_{20}=9$ 
et en  utilisant l'\'egalit\'e~\eqref{(8.1)}, on v\'erifie   que
$T_{152}\equiv 74\pmod{151}$. L'\'equivalence annonc\'ee est donc vraie dans ce cas. 

Soit $p$ un diviseur premier de $B_f$ distinct de $2,5,19,151$. En utilisant l'\'equivalence~\eqref{(1.2)}, on constate que l'on a $N_p(f)\neq 5$. Par ailleurs, on a
$$T_8=4,\quad T_{468}\equiv 250\pmod{467}, \quad T_{762}\equiv 355\pmod{761}\quad \text{et}\quad T_{2478}\equiv 695\pmod{2477}.$$
En particulier, on a $T_{p+1}\not \equiv 0 \pmod p$, 
 d'o\`u  le r\'esultat (th\'eor\`eme~\ref{T:thm1}). 
 
 On  notera que l'on a $T_3=T_6=0$ et $N_2(f)=N_5(f)=0$, donc les nombres premiers $2$ et $5$ doivent \^etre exclus.

2) On a $\Delta_f=67.743$.
La d\'emarche utilis\'ee pour d\'emontrer la premi\`ere assertion  permet \`a  nouveau de conclure. En d\'eterminant  une approximation  num\'erique des racines de $f$, on v\'erifie que la d\'ecomposition de $B_f$ en produit de nombres premiers est donn\'ee par l'\'egalit\'e

{\small{
$$B_f=2^{72}.3^4.7^2.13^2.41^2.47^5.67.79^3.281^2.347^2.743.1151^4.1283^2.1319^2.1663^2.951697.1395487^2.$$
$$6367393^2.17122219.136254761^2.57785936129^2.123530989187^2.885187290897569369^2.$$}}

En appliquant  l'\'equivalence~\eqref{(1.2)} et l'\'egalit\'e~\eqref{(8.1)} avec chaque diviseur premier de $B_f$, on obtient  le r\'esultat annonc\'e.

3) On a $\Delta_f=-776887$. Dans ce cas, en proc\'edant comme ci-dessus, on constate que l'entier $B_f$  poss\`ede 1723 chiffres d\'ecimaux et  nous ne sommes pas parvenus \`a  le factoriser directement.  Afin d'obtenir cette factorisation, on a utilis\'e l'\'egalit\'e~\eqref{(2.9)}. Pour chaque classe de conjugaison $\calS$ de $\sym_7$, on a donc
explicit\'e l'entier $B_{\calS}$. Pour chaque classe $\calS$, on adopte la notation  $B_{\calS}=B_v$, dans laquelle $v=[\ell_1,\cdots,\ell_t]$ est un vecteur dont les composantes 
sont les longueurs des cycles intervenant dans la d\'ecomposition 
des \'el\'ements de $\calS$ en produit de cycles \`a supports disjoints. On obtient les r\'esultats suivants :

{\small{
$$B_{[7]}=\left(2^{25}. \ 7^{66}. \ 10962571225722870541309365904873297427\right)^2,$$
$$B_{[1,6]}=\left(7^{63}. \ 5087. \ 615078503681.\ p.q\right)^2,$$
o\`u
$$ p=2315322299227184940410117, \quad q=103180663032729967322136080457913269014828964041,$$
$$B_{[2,5]}=\left(3^3.\  7^{42}.54401205406822254534362466407.  \ 1246534314610754363757777242593\right)^2,$$
$$B_{[3,4]}=\left(31.\ 58435252078103192479377043961\right)^4,$$
$$B_{[1,1,5]}=\left(3.\  7^{21}. \ 107. \ 109622861.\ 314517883. \ 87453951749. \ 96257721299. \ 473705767399763\right)^2,$$
$$B_{[1,2,4]}=\left(2^{21}. \ 7^{42}. \ 29^2. \ 107. \  15961129. \ 24534049. \ 198331229. \ 671794760853523.  \ p.q\right)^2,$$
$$\text{o\`u} \quad p=93177762039493501,\quad q=10900667110067270212049432531,$$
$$B_{[1,3,3]}=\left(2^{28}.\ 7^{7}. \  1085687 \right)^4,$$
$$B_{[2,2,3]}=\left(31.\ 13132283. \ 161620073077054859. \ 183574845951173009\right)^2,$$
$$B_{[1,1,1,4]}=\left(4936189. \  725938918439654319174389 \right)^2,$$
$$B_{[1,1,2,3]}=\left(7^{42}.\ 32717. \ 43670581. \  4063646878656760059708736369066517857\right)^2,$$
$$B_{[1,2,2,2]}=-7^{21}.\ 761. \ 7679513. \ 25839993284328785428639,$$
$$B_{[1,1,1,1,3]}=7^{28},$$
$$B_{[1,1,1,2,2]}=-2^{21}.\ 17.\ 191.\ 5087. \ 15031. \ 28627874657408393618159298227,$$
$$B_{[1,1,1,1,1,2]}=776887.$$
}}

Pour $p=776887$, qui est au signe pr\`es le discriminant de $f$, on a $T_{p+1}\equiv 115287\pmod p$.  En ce qui concerne les autres diviseurs  premiers des $B_{[v]}$,
l'\'equivalence~\eqref{(1.2)} et l'\'egalit\'e~\eqref{(8.1)} permettent alors d'obtenir le r\'esultat.

Notons que pour $p=15961129$, on a $N_p(f)=1$ et $T_{p+1}\equiv 0 \pmod p$, donc ce nombre premier est \`a exclure.
Par ailleurs, on constate que,  conform\'ement \`a la proposition~\ref{P:prop1},  les entiers $B_{\calS}$ pour lesquels $\calS$ n'est pas form\'ee d'\'el\'ements d'ordre $2$ sont des carr\'es. 
Cela termine la d\'emonstration de la proposition~\ref{P:prop4}.

 \subsection{La proposition 5} On a $\Delta_f=47^2$ et  $T_2=0$, $f$ est irr\'eductible sur $\Q$ et  $\Gal(f)$ est di\'edral d'ordre $10$  (cf.  \cite{lmfdb}). On v\'erifie que l'on a
$$\prod_{\calS}F_{\calS}(\alpha_1,\alpha_2,\alpha_3,\alpha_4,\alpha_5)=-3^2.5^{10}.11^4.13^4.19^2.23^3.41^2.47^3.281^2,$$
 o\`u $\calS$ parcourt l'ensemble des classes de conjugaison non triviales de $\sym_5$.  D'apr\`es la condition~\eqref{(2.12)} l'ensemble  des diviseurs premiers de $B_f$ 
 est donc contenu dans 
 $$\lbrace 3, 5, 11, 13, 19, 23, 41, 47, 281\rbrace.$$
 Le th\'eor\`eme~\ref{T:thm3} entra\^ine alors le r\'esultat.

  \subsection{La proposition 6} On a 
  $$\Delta_f=18441\alpha^2 - 621100\alpha + 1031256.$$
  Il y a deux id\'eaux premiers dans $O_K$ au-dessus de $5$ et on a une \'egalit\'e de la forme  $5O_K=\fp_5\fq_5$,  o\`u le degr\'e r\'esiduel de $\fp_5$ vaut $1$ et o\`u celui de $\fq_5$ vaut $2$. 
  De plus, on a $\fp_5=(u)$.  Il en r\'esulte que $f$ est un polyn\^ome d'Eisenstein en l'id\'eal $\fp_5$. En particulier, $f$ est irr\'eductible sur $K$.  
  
  Par ailleurs, $13$ est inerte dans $K$ et $f$ se d\'ecompose   en deux facteurs irr\'eductibles de degr\'e  2 et 3 modulo $13O_K$. Il en r\'esulte que $\Gal(f)$ contient une transposition. Le polyn\^ome $f$ \'etant  irr\'eductible sur $K$, $\Gal(f)$ contient un $5$-cycle. On en d\'eduit que  $\Gal(f)=\sym_5$. 
  
 Pour chaque classe de conjugaison $\calS$ de $\sym_5$,  on note,  comme dans l'alin\'ea 3 ci-dessus,  $B_{\calS}=B_v$, o\`u  $v=[\ell_1,\cdots,\ell_t]$ est un vecteur indiquant les longueurs des cycles intervenant dans la d\'ecomposition 
des \'el\'ements de $\calS$ en produit de cycles \`a supports disjoints. On obtient les r\'esultats suivants :
{\small{
$$B_{[5]}=(-413075\alpha^2 + 640600 \alpha - 39900)^2,$$
$$B_{[1,4]}=(48990625\alpha^2 - 38715625\alpha - 41359375)^2,$$
$$B_{[2,3]}=(16516\alpha^2 - 187975\alpha + 309206)^2,$$
$$B_{[1, 1, 3]}=(-14041\alpha^2 - 368900\alpha + 619144)^2,$$
$$B_{[1, 2, 2]}=31472487500\alpha^2 - 24902178125\alpha - 26712984375,$$
$$B_{[1, 1, 1, 2]}=\Delta_f.$$
}}

Notons $\Norm_{K/\Q}(B_v)$ la norme de $K$ sur $\Q$ de $B_v$. On v\'erifie que l'on a
{\small{
$$\Norm_{K/\Q}(B_{[5]})=5^{24},\quad \Norm_{K/\Q}(B_{[1,4]})=-5^{32},\quad  \Norm_{K/\Q}(B_{[2,3]})=5^9.181.307.167449,$$
$$ \Norm_{K/\Q}(B_{[1, 1, 3]})=5^9.2707.15639581,\quad \Norm_{K/\Q}(B_{[1, 2, 2]})=-5^{26}.61.70956089917,$$
$$\Norm_{K/\Q}(\Delta_f)=5^9.23.367.1613.20101.$$
 }}

Pour chacun des  id\'eaux premiers de $O_K$   divisant  les $B_v$, il  s'agit alors de v\'erifier s'ils satisfont ou non  l'\'equivalence \eqref{(2.17)}. D'apr\`es les \'egalit\'es pr\'ec\'edentes, mis \`a part les id\'eaux premiers $\fp_5$ et $\fq_5$ au-dessus de $5$, ils sont tous de degr\'e r\'esiduel $1$. 
 
 En ce qui concerne  $\fp_5$ et $\fq_5$, on a $\Norm(\fp_5)=5$ et $\Norm(\fq_5)=25$. On constate que $T_6=0$ et que $T_{26}$  appartient \`a  $\fq_5$. Par ailleurs, on a $\alpha\equiv 2 \pmod{\fp_5}$, d'o\`u $u\in \fp_5$ et  $f$ modulo $\fp_5$ est $X^5$. On a donc $N_{\fp_5}(f)=1$ et $\fp_5$ est  \`a exclure. On a $u\equiv 3 \alpha \pmod{\fq_5}$ donc 
$f$ modulo $\fq_5$ est $X^5+(7\alpha+2)X+3\alpha$. On v\'erifie alors que l'on a $N_{\fq_5}(f)=5$, ainsi  $\fq_5$ n'est pas \`a exclure. 
 
Examinons  l'\'equivalence \eqref{(2.17)} pour les autres id\'eaux premiers de $O_K$ qui divisent $B_{[2,3]}$.  Il existe un unique id\'eal premier $\fp$ de $O_K$ de degr\'e  $1$ au-dessus de $181$.  On a $\alpha\equiv 30 \pmod {\fp}$, d'o\`u l'on d\'eduit que $T_{182}\equiv 57 \pmod {\fp}$ et $N_{\fp}(f)=3$, donc  $\fp$ n'est pas exclure.  Par ailleurs,  il y a trois id\'eaux premiers de degr\'e $1$ au-dessus de $307$ et un seul  divise $B_{[2,3]}$. Notons-le $\fq$. On a $\alpha\equiv 100 \pmod{\fq}$, d'o\`u $T_{308}\equiv 255 \pmod{\fq}$ et $N_{\fq}(f)=1$, ainsi $\fq$ n'est pas \`a exclure. De m\^eme,  il y a trois id\'eaux premiers de degr\'e $1$ au-dessus de $167449$. En notant $\fq$ celui qui divise  $B_{[2,3]}$, on a $\alpha\equiv 30636 \pmod{\fq}$, d'o\`u $T_{167450}\equiv 51379\pmod{\fq}$, $N_{\fq}(f)=2$  et la m\^eme conclusion.
 
 Les autres diviseurs premiers  des $B_v$ se tra\^itent de fa\c con analogue. 
 
 Pour l'unique id\'eal premier $\fp_{61}$ de $O_K$ de degr\'e $1$ au-dessus de $61$, on a $\alpha\equiv 57\pmod{\fp_{61}}$, d'o\`u $T_{62}\in \fp_{61}$ et $N_{\fp_{61}}(f)=1$, donc $\fp_{61}$ ne satisfait pas l'\'equivalence \eqref{(2.17)}, d'o\`u le r\'esultat.
 
\section{Remarque sur les id\'eaux premiers principaux d'un corps de nombres \label{S:par9}}
Consid\'erons un corps de nombres $K$.  
 Les id\'eaux premiers non nuls de $O_K$ qui sont principaux sont exactement ceux qui sont totalement d\'ecompos\'es dans le corps de classes de Hilbert $H_K$ de $K$ (\cite{Neu}, (8.5) Corollary, p. 107).
Connaissant $H_K$, les r\'esultats pr\'ec\'edent permettent donc d'obtenir  un crit\`ere caract\'erisant les id\'eaux premiers principaux de $O_K$. 

Soit $f\in O_K[X]$ le  polyn\^ome minimal d'un \'el\'ement primitif entier de l'extension $H_K/K$.   Soit $(T_n)_{n\in \N}$ la suite d'\'el\'ements de $O_K$ associ\'ee \`a $f$ par   l'\'egalit\'e~\eqref{(2.1)}.

 On se limirera ici \`a expliciter deux exemples illustrant cette situation. 
 
 1)  Prenons $K=\Q\left(\sqrt{-5}\right)$ dont le nombre de classes vaut $2$.  On a $H_K=K\left(\sqrt{-1}\right)$, d'o\`u $f=X^2+1$ et $\Delta_f=-20$. Il en r\'esulte que pour tout id\'eal premier $\fp$ de $O_K$,  on a $N_{\fp}(f)=2$ si et seulement si $\fp$ est principal ({\it{loc. cit.}}  et \cite{Cohen}, Prop. 2.3.9). Par suite, 
 \begin{equation}
 \label{(9.1)} 
\fp\ \text{est principal}  \   \Leftrightarrow \  \Norm(\fp)\equiv 1 \pmod 4.
  \end{equation}
  
On peut retrouver  ce fait   en remarquant que  la  suite $(T_n)_{n\in \N}$   associ\'ee \`a $f$ est d\'efinie par  
$$T_n=2\quad  \text{si}\ n\equiv 0 \pmod 4,\quad T_n=-2\quad  \text{si}\ n\equiv 2 \pmod 4,\quad T_n=0\quad  \text{si}\ n\equiv 1 \pmod 2.$$
D'apr\`es la proposition~\ref{P:prop2},  pour tout   id\'eal premier $\fp$ de $O_K$,  on a donc 
$N_{\fp}(f)=2$ si et seulement si $\Norm(\fp)\equiv 1 \pmod 4$, d'o\`u l'\'equivalence~\eqref{(9.1)}.

 2) Prenons  $K=\Q(\beta)$ avec  $\beta^4+7\beta^2-2\beta+14=0.$
 Le groupe de Galois sur $\Q$ de  la cl\^oture galoisienne de $K$ est $\A_4$. 
 Le groupe des classes de $K$ est cyclique d'ordre $4$ i.e. on a  $\Gal\left(H_K/K\right)=\Z/4\Z.$  En utilisant  {\tt Magma},  on constate que le polyn\^ome  
$f=X^4-(\beta^2+3)X^2-1$
convient. La suite $(T_n)_{n\in \N}$ associ\'ee \`a $f$ est d\'efinie par les \'egalit\'es
$$T_0=4,\quad T_1=0,\quad T_2=2(\beta^2+3),\quad T_3=0,
\quad T_{n+4}=(\beta^2+3)T_{n+2}+T_n\quad  \text{pour tout}\ n\in \N.$$
 
 \begin{lemme} \label{L:lemme8}
 Soit $\fp$ un id\'eal premier de $O_K$ de caract\'eristique r\'esiduelle impaire. Alors,
 \begin{equation}
  \label{(9.2)} 
\fp\ \text{est principal}  \   \Leftrightarrow \  T_{\Norm(\fp)+1}\equiv 2(\beta^2+3) \pmod {\fp}.
  \end{equation}
\end{lemme}

\begin{proof} Il y a deux id\'eaux premiers $\fp_1$ et $\fp_2$ de $O_K$ au-dessus de $2$, chacun d'indice de ramification $2$,  et deux id\'eaux premiers de $O_K$ au-dessus de $5$ dont un seul $\fp_5$ est de degr\'e r\'esiduel $1$. On v\'erifie que l'on a
 \begin{equation}
  \label{(9.3)}
\Delta_fO_K=\fp_1^8\fp_2^{16}\fp_5^4.
  \end{equation}
  Par ailleurs, on a 
  $$e_1=0,\quad e_2=\beta^2+3,\quad e_3=0, \quad e_4=-1.$$
 Les seuls \'el\'ements d'ordre $4$ de $\sym_4$ sont les $4$-cycles. On d\'eduit alors du th\'eor\`eme~\ref{T:thm3} que $B_f$ divise l'entier $R_{\calS_3}$ d\'efini par l'\'egalit\'e~\eqref{(7.5)}.
On a 
$$R_{\calS_3}=4e_2^2(9e_2^2+4)^2=-9\beta^2+18\beta-41.$$
On constate que l'on a 
$$(-9\beta^2+18\beta-41)O_K=\fp_2^4\fq,$$
o\`u $\fq$ est l'id\'eal premier de $O_K$ au-dessus de $80233$ de degr\'e r\'esiduel $1$. D'apr\`es le  th\'eor\`eme~\ref{T:thm1}, 
pour tout id\'eal premier $\fp$ de $O_K$, distinct de $\fp_1$, $\fp_2$, $\fp_5$ et $\fq$, l'\'equivalence~\eqref{(9.2)} est satisfaite.

Les id\'eaux $\fp_1$ et $\fp_2$ ne sont pas principaux. On a les \'egalit\'es $\Norm(\fp_1)=\Norm(\fp_2)=2$ et $T_3=0$,  donc pour ces deux id\'eaux $\fp_i$ on a 
$T_{\Norm(\fp_i)+1}\equiv 2(\beta^2+3) \pmod {\fp_i}$.
L'id\'eal $\fp_5$ n'est pas principal et  on a $T_6\not\equiv 2(\beta^2+3)\pmod{\fp_5}$.
Par ailleurs, $\fq$ est principal et on a $T_{\Norm(\fq)+1}\equiv 2(\beta^2+3) \pmod {\fq}$,  d'o\`u le r\'esultat.
\end{proof}

\begin{remark} Par exemple, l'id\'eal premier $\fp_{13}$ de $O_K$ au-dessus de $13$ de degr\'e 1 est principal, on a 
$\fp_{13}=(\beta^3/2 - \beta^2/2 + 2\beta - 4)$ et on v\'erifie que l'on a 
$v_{\fp_{13}}\left(T_{14}-2(\beta^2+3)\right)=1.$ C'est l'id\'eal premier de $O_K$  de plus petite norme qui soit principal.
\end{remark}

\end{document}